\documentclass[10pt]{amsart}
\usepackage{graphicx}
\usepackage{latexsym}
\usepackage{fancyhdr}
\usepackage{amsmath, amssymb}
 \usepackage[utf8]{inputenc}
\usepackage[all]{xy}
\usepackage{pdflscape}
\usepackage{longtable}
\usepackage{rotating}
\usepackage[polish, english]{babel}
\usepackage{lmodern}
\usepackage[T1]{fontenc}
\usepackage{soul}
\usepackage{subfigure}
\usepackage{mathrsfs}
\usepackage{stmaryrd}
\usepackage{tensor}
\usepackage{enumerate}
\usepackage{hyperref}

\usepackage{color}
\definecolor{purple}{rgb}{0.5, 0.0, 0.5}

\theoremstyle{plain}
\newtheorem{theorem}{Theorem}[section]

\newtheorem{lemma}[theorem]{Lemma}

\theoremstyle{definition}
\newtheorem{definition}[theorem]{Definition}
\theoremstyle{remark}
\newtheorem{remark}[theorem]{Remark}
\newtheorem{example}[theorem]{Example}

\numberwithin{equation}{section}

\newcommand{\lgw}{\longrightarrow}
\newcommand{\lgm}{\longmapsto}
\newcommand{\si}{\sigma}

\newcommand{\wdh}{\widehat}

\newcommand{\D}{\Delta}

\newcommand{\ord}{\operatorname{ord}}

\newcommand{\Z}{\mathbb{Z}}
\newcommand{\kk}{\Bbbk}
\newcommand{\R}{\mathbb{R}}
\newcommand{\K}{\mathbb{K}}

\newcommand{\N}{\mathbb{N}}
\newcommand{\C}{\mathbb{C}}
\newcommand{\Q}{\mathbb{Q}}
\renewcommand{\t}{\tau}
\renewcommand{\lg}{\langle}
\newcommand{\rg}{\rangle}
\newcommand{\lb}{\llbracket}
\newcommand{\rb}{\rrbracket}

\renewcommand{\phi}{\varphi}

\renewcommand{\o}{\omega}

\let\mathscr\mathcal

\begin{document}
\title[P{\l}oski Approximation Theorem]{P{\l}oski Approximation Theorem}

\author{Adam Parusi\'nski}
\email{adam.parusinski@univ-cotedazur.fr}
\address{Universit\'e Côte d'Azur, CNRS, LJAD, UMR 7351, 06108 Nice, France}

\dedicatory{To the memory of Arkadiusz P{\l}oski}

\author{Guillaume Rond}
\email{guillaume.rond@univ-amu.fr}
\address{Aix-Marseille Universit\'e, CNRS, Centrale Marseille, I2M, UMR 7373, 13453 Marseille, France}

\thanks{This research was funded, in whole or in part, by l’Agence Nationale de la Recherche (ANR), project ANR-22-CE40-0014. For the purpose of open access, the author has applied a CC-BY public copyright licence to any Author Accepted Manuscript (AAM) version arising from this submission.}


\keywords{Artin approximation, Zariski equisingularity, Algebraization of analytic and meromorphic function germs.}

\subjclass[2010]{
14B12, 
32B10,    
32S15,  
14B05. 
32S05. 
}

\begin{abstract}
  The aim of this paper is to review how some approximation results in commutative algebra are being used to construct equisingular deformations of singularities. The first example of such an approximation result appeared for the first time in A. P\l oski's PhD thesis.
\end{abstract}

\maketitle


\section{Introduction}
Artin’s seminal approximation theorem \cite{Ar1}, which established the density of convergent solutions in the space of formal solutions for systems of analytic equations, is an important result with numerous  applications in different fields, in particular in singularity theory \cite{Ro}. This theorem lead to new problems of approximations in local algebra and local analytic geometry \cite{Ar2, Ar70}.

In his PhD thesis, A. P\l oski made a significant contribution to this area by proposing a strengthened version of Artin's theorem, now known as the P\l oski's approximation theorem. His result not only extended Artin's theorem but also gave hints for generalizations in commutative algebra and provided applications to singularity theory. In particular, it has implications for the study of equisingular deformations of singularities, which play a key role in understanding the topological and analytical properties of singular spaces.

This paper reviews P\l oski's approximation theorem, examining its implications and generalizations, particularly in the context of equisingular deformations and algebraization of analytic sets and functions. By integrating recent advancements, we aim to shed light on the broader relevance of P\l oski's theorem. Furthermore, we explore their connections to the algebraization of function germs and meromorphic functions, with a focus on two variables cases. These results demonstrate the far-reaching impact of P\l oski's work on contemporary mathematical research.

\section{Approximation Theorems}
\subsection{Artin approximation theorem}
In 1968 M. Artin proved the following result:

\begin{theorem}[Artin approximation Theorem]\label{thm:artin-app}\cite{Ar1}
Let $x=(x_1,\ldots, x_n)$ and $y=(y_1,\ldots, y_p)$ be an $n$-tuple and a $p$-tuple of indeterminates. Let $f=(f_1,\ldots, f_m)\in\C\{x,y\}^m$ be an $m$-tuple of convergent power series. Assume given a formal power series solution vector $\wdh y(x)\in\C\lb x\rb^p$ vanishing at 0:
$$f(x,\wdh y(x))=0.$$
Then, for every $c\in\N$, there is a convergent power series solution vector $y(x)\in\C\{x\}^p$ vanishing at 0:
$$f(x,y(x))=0$$
such that 
$$\forall i,\qquad y_i(x)-\wdh y_i(x)\in (x)^c.$$
\end{theorem}

This result means that the set of convergent solutions of $f(x,y)=0$ is dense in the set of formal solutions for the $(x)$-adic topology.
One year later, M. Artin proved several extensions of this theorem. In particular he proved that this result remains true if we replace the rings of convergent power series $\C\{x\}$ and $\C\{x,y\}$ by the rings of algebraic power series $\kk\lg x\rg$ and $\kk\lg x,y\rg$ for a characteristic zero field $\kk$. Let us recall that the elements of $\kk\langle x\rangle$ are the formal power series $f(x)$ such that $P(x,f(x))=0$ for some nonzero polynomial $P(x,t)\in\kk[x,t]$. Moreover this set of algebraic power series is a Noetherian local ring satisfying the Weierstrass division and preparation theorems.

\subsection{P\l oski approximation theorem}

In his PhD thesis \cite{ploski1973}, A. P\l oski made a deep study of  the method of M. Artin to prove a strengthened version of his theorem:

\begin{theorem}[P\l oski approximation theorem]\label{thm:ploski-app}
Let $x=(x_1,\ldots, x_n)$ and $y=(y_1,\ldots, y_p)$ be an $n$-tuple and a $p$-tuple of indeterminates. Let $f=(f_1,\ldots, f_m)\in\C\{x,y\}^m$ be an  $m$-tuple of convergent power series. Assume given a formal power series solution vector $\wdh y(x)\in\C\lb x\rb^p$ vanishing at 0:
$$f(x,\wdh y(x))=0.$$
Then there is a convergent power series solution $y(x,z)\in\C\{x,z\}^p$, where $z=(z_1,\ldots, z_s)$ is a new $s$-tuple of indeterminates:
$$f(x,y(x,z))=0$$
and a vector of formal power series $\wdh z(x)\in\C\lb x\rb^s$ vanishing at $0$ such that
$$y(x,\wdh z(x))=\wdh y(x).$$
\end{theorem}

This result trivially implies Artin approximation theorem. Indeed, with the notations of Theorem \ref{thm:artin-app} and Theorem \ref{thm:ploski-app}, it is enough to replace $\wdh z(x)$ by any vector of convergent power series $z(x)$ such that $\wdh z(x)-z(x)\in (x)^c$, and set $y(x):=y(x,z(x))$ to obtain the conclusion of Theorem \ref{thm:artin-app}.

A. P\l oski published this theorem in \cite{ploski1974} without the details of the proof. He did not published any complete proof (except of his PhD thesis \cite{ploski1973} written in Polish) before 2017. At a conference at Lille in 1999, he gave a course presenting this result, and wrote notes with a complete proof \cite{ploski99} (in French). These notes are available on the web, but have not been formally published. Therefore in 2017, A. P\l oski published a complete proof in \cite{ploski17}.

\begin{remark}
P{\l}oski's result can be roughly rephrased as follows: any formal power series solution of $f(x,y)=0$ is a formal point in an analytically parametrized family of solutions. This result can be thought as a kind of uniformization theorem.
\end{remark}

\begin{remark}
P{\l}oski's proof works in exactly the same way alter replacing rings of convergent series with rings of algebraic series over a  characteristic zero field. Indeed, the essential arguments of the proof are the Jacobian criterion and the Weierstrass division theorem. In fact, this theorem holds in the more general framework of Weierstrass systems in any characteristic \cite{RoW}.
\end{remark} 

\begin{example}
 Assume that $f(x,y)$ satisfies the assumptions of the Implicit function Theorem : $f(0,0)=0$ and $\frac{\partial f}{\partial y}(0,0)$ is an invertible matrix (here we assume $m=p$). Then the equation $f(x,y)=0$ has only one solution $y^0(x)\in\C\{x\}^p$ which is convergent. Therefore Theorem \ref{thm:ploski-app} is obviously true with $s=0$ and $y(x)=y^0(x)$.
\end{example}

\begin{example}
 Consider one equation $f(x,y):=y_1^2-y_2^3=0$. Then, since $\C\{x\}$ are $\C\lb x\rb$ are unique factorization domains, the set of convergent (resp. formal) solutions is
$$\{(z(x)^3,z(x)^2)\mid z(x)\in\C\{x\}\ (resp. \C\lb x\rb)\}.$$
Therefore Theorem \ref{thm:ploski-app} is true with $s=1$. Moreover, in this case the above analytically parametrized family of solutions is the whole set of solutions of $f(x,y)=0$.
\end{example}

\begin{example}
Consider the same $f(x,y)$ as before and assume that $x$ is a single indeterminate. Let $(\wdh y_1(x), \wdh y_2(x))$ be a nonzero formal power series solution. Let $d:=\ord(\wdh y_1(x))>0$. Since $\wdh y_1(x)^2=\wdh y_2(x)^3$, we see that $d\in 3\Z$. Let $d=3(e+1)$ where $e\in\N$. Then we have $\wdh y_1(x)=x^{3e}\wdh z(x)^3$ and $\wdh y_2(x)=x^{2e}\wdh z(x)^2$ for some formal power series $\wdh z(x)$ vanishing at 0. Then in Theorem \ref{thm:ploski-app} we can choose $(y_1(x,z),y_2(x,z)):=(x^{3e}z^3, x^{2e}z^2)$. The analytically parametrized family of solutions we obtain in this way does not cover the whole set of solutions of $f(x,y)=0$.\\
P{\l}oski's proof is completely effective, and if we follows his proof for this particular example, we obtain exactly this family of solutions.  Therefore,  P{\l}oski's proof does not provide the whole set of solutions of $f(x,y)=0$ in general.
\end{example}

\begin{remark}
Giving a formal power series solution $\wdh y(x)$ of $f(x,y)=0$  is equivalent to the data of a morphism of $\C\{x\}$-algebra:
$$\phi:\frac{\C\{x\}[y]}{(f(x,y))}\lgw \C\lb x\rb$$
Theorem \ref{thm:ploski-app} asserts that this morphism factors as
 $$\xymatrix{ & \C\llbracket x\rrbracket\\
 \dfrac{\C\{x,y\}}{(f(x,y))} \ar[ru]^{\phi}  \ar[r]& \C\{x,z\} \ar[u]}$$  
 \end{remark}
 
Having taken note of P\l oski's Theorem, D. Popescu conjectured that this result was more general: given a Henselian local ring $A$, $\wdh A$ denoting its completion, for every finitely generated $A$-algebra $B$ and any $A$-morphism $\phi:B\lgw \wdh A$, there exists a \emph{smooth} $A$-algebra $C$ such that $\phi$ factors through $C$:
 $$\xymatrix{ & \wdh A\\
B \ar[ru]^{\phi}  \ar[r]& C\ar[u]}$$  
Then D. Popescu proved his conjecture in \cite{Pop} (see also \cite{Rot}, \cite{Sp}, \cite{Qu}, \cite{Sw} for subsequent proofs).  Moreover, in \cite{Sp}, M. Spivakovsky proved a nested version of Popescu's Theorem. We will not state here this result, whose formulation is a bit technical, but from Spivakovsky's theorem we can deduce the following nested version of P{\l}oski's  result:

\begin{theorem}[Nested Artin-P{\l}oski-Popescu approximation theorem, \cite{BPR}]\label{thm:nested_ploski}
Let $f(x,y)\in \kk \lg x\rg[y]^m$ and consider a solution $y(x)\in\kk\lb x\rb^p$ of $$f(x,y(x))=0.$$
Assume that $y_i(x)$ depends only on $(x_1,\ldots, x_{\si(i)})$ where $i\lgm \si(i)$ is an increasing function. Then there exist a new set of indeterminates $z=(z_1,\ldots, z_s)$, an increasing function $\t$, convergent power series $z_i(x)\in\kk \lb x\rb$ vanishing at 0 such that $z_1(x)$, \ldots, $z_{\t(i)}(x)$ depend only on $(x_1,\ldots, x_{\si(i)})$, and  an  algebraic power series vector solution $y(x,z)\in\kk \lg x,z\rg^p$ of
 $$f(x,y(x,z))=0$$ such that for every $i$, 
 $$y_i(x,z)\in\kk \lg x_1,\ldots,x_{\si(i)},z_1,\ldots,z_{\t(i)}\rg \text{ and }y(x)=y(x,z(x)).$$
\end{theorem}

\begin{remark}
Let us mention that this nested version of P{\l}oski's theorem is no longer true for rings of convergent power series in view of an example due to A. Gabrielov \cite{Gab}.
\end{remark}

\section{Algebrization of the germ of an analytic set or of an analytic function}
Let $\K=\R$ or $\C$, $x=(x_1, \ldots , x_n)$. If $f\in\K\{x\}$ is an isolated singularity then, by a result of Samuel \cite{Sa}, $f$ is \emph{finitely determined}, that is, there exists an integer $k$ such that for every $g\in\K\{x\}$ with $f-g\in(x)^k$, there exists an analytic diffeomorphism $h:(\K^n,0)\lgw (\K^n,0)$ such that $$f\circ h(x)=g(x).$$ 
In particular, if we choose $g$ to be the truncation of $f$ at an order $\geq k$, we find that $f$ can be transformed into a polynomial by a local analytic change of coordinates (one says that $f$ is analytically equivalent to a polynomial function germ). Several authors generalized Samuel's result, and eventually Kucharz \cite{Ku} proved that every (non necessarily reduced) analytic function $f$ in two variables is \emph{weakly finitely determined}. This means that if  $f=f_1^{\ell_1}\cdots f_p^{\ell_p}$ is the decomposition of $f$ into a product of irreducible factors, then there is an integer $k$ such that, for every $g_i\in\K\{x\}$ with $f_i-g_i\in (x)^k$, there exists an analytic diffeomorphism $h:(\K^n,0)\lgw (\K^n,0)$ such that $$f\circ h(x)=g_1^{\ell_1}\cdots g_p^{\ell_p}.$$
Weakly finitely determined germs are also equivalent to a polynomial function germ. Kucharz also proved that any analytic function in $n$ variables is equivalent to a polynomial in two variables whose coefficients are analytic functions in $n-2$ variables.

Therefore, in the case of $n=2$, any analytic function germ  can be transformed by an analytic change of coordinates to a polynomial one. 
When $f$ is a (reduced) convergent power series in three or more variables, this is no longer true: H. Whitney \cite{W} gave an example of a three-variable reduced convergent power series which is not equivalent to a polynomial function or to an algebraic one.

We explain below how to construct a deformation  of a given analytic function germ $f: (\K^n,0)\to (\K,0)$ that is topologically trivial (with respect to the right equivalence, i.e. by a homeomorphism in the source $(\K^n,0)$), and such that one fiber of this deformation is an algebraic function germ. This construction is based on Theorem \ref{thm:nested_ploski} and the notion of Zariski equisingularity.   

\subsection{Zariski equisingularity}
Zariski equisingularity of families of singular varieties was introduced by Zariski in \cite{Za} in the context of equisingularity of a hypersurface along a smooth subvariety.    It can be formulated over any field of characteristic zero in the algebroid set-up (varieties defined by the formal power series) and over $\K=\R$ or $\C$ in the analytic case.  For a survey on Zariski equisingularity see \cite{Par}, which also contains an appendix on higher order (also called generalized) discriminants that are used in the construction below. 



\begin{definition}
Let $V= F^{-1}(0)$, $F\in \K\{t,x\}$, be an analytic hypersurface in a neighborhood 
of the origin in $\K^\ell \times \K^n$.  
We say that $V$ is \emph{Zariski equisingular with respect to the 
parameter $t\in \K^\ell$ (and a local system of coordinates $x_1, \ldots, x_n$ in $\K ^n$)}
if (with $\pi:(\K^{\ell}\times\K^n,0)\lgw (\K^{\ell}\times\K^{n-1},0)$ denoting the canonical projection):
\begin{enumerate}
\item $ (\K^\ell \times \{0\},0)  \subset (V,0)$;
    \item The restriction $\pi_{|_{(V,0)}}$ is (algebraically) finite;
    \item The branch locus of $\pi_{|_{(V,0)}}$ is itself Zariski equisingular with respect to the 
parameter $t$;
    \item When $\ell=\dim(V,0)$, then $(V,0)$ is Zariski equisingular with respect to the 
parameter $t$ if $(V,0)=(\K^\ell,0)$ (or at some stage the branch locus is empty).
\end{enumerate}
We say that $V$ is \emph{Zariski equisingular with respect to the 
parameter $t\in \K^\ell$} if it is so after a local change of coordinates 
$(\K^\ell \times  \K^n,0) \to (\K^\ell \times  \K^n,0)$ preserving the parameter $t$.  
\end{definition}

\begin{remark}
Note that this notion depends heavily on the local choice of coordinates.  Therefore often we consider Zariski equisingularity with respect to generic or generic linear $x$ coordinates \cite{zariski1979}. 
Whether a generic linear choice is generic in the sense of \cite{zariski1979} is an open problem for singularities of codimension $\ge 3$.  For singularities of codimension $1$, i.e. equisingular families of plane curves the positive answer follows from Zariski's theory of equisingularity of families of plane curves.  For singularities in codimension $2$, i.e. equisingular families of surfaces singularities in $\C^3$, the positive answer was given in  \cite{PPJAG2024}.   
\end{remark}

In practice one may argue as follows.  Let $F$ be an analytic function defining $(V,0)$ and suppose that $F(0,x)\not \equiv 0$. Then after a linear change of coordinates $x$, we may assume that $F$ is a pseudopolynomial $F_n$ times a unit $u_n(t,x)$. This means that  $F_n$ is a polynomial in $x_n$ with coefficients that are analytic in  $(t,x^{n-1})$ (Here $x^{n-1}=(x_1,\ldots, x_{n-1})$):
$$F_n(t,x)=x_n^{p_n}+\sum_{j=1}^{p_n}a_{n-1,j}(t,x^{n-1})x_n^{p_n-j}$$
where $a_{n-1,j}(0)=0$ for every $j$.
 Then we denote by $\Delta_n(t,x^{n-1})$ the discriminant of $F_n$ seen as a polynomial in $x_n$. If $F$ is reduced, $\Delta_n$ is not identically zero. We assume again that $\Delta_n(0,x)\not \equiv 0$.  In general, if $F$ is not reduced we replace it by $(F)_{red}$, or equivalently we consider the higher order discriminants of $F_n$. 
Thus by induction we define a sequence of pseudopolynomials
 \begin{align*}
F_{i} (t, x^i )= x_i^{p_i}+ \sum_{j=1}^{p_i} a_{i-1,j} (t,x^{i-1})
 x_i^{p_i-j}, \qquad    i=0, \ldots ,n,
\end{align*}
$t\in \K ^ \ell$,  $x^i :=(x_1,\ldots, x_i)\in\K^i$, 
with analytic coefficients $a_{i-1,j}$, that satisfy 
\begin{enumerate}
\item
$F_{i-1} (t, x^{i-1})=0$  if and only if $F_{i} (t, x^{i-1}, x_i)=0$ considered as an equation in $x_i$ with 
$(t, x^{i-1})$ fixed, has fewer complex (!) roots than for generic $(t, x^{i-1})$; 
\item
$F_0\equiv 1$; 
\item
there are positive reals $\delta_k>0$, $k=1, \ldots, \ell$, and $\varepsilon _j> 0$, $j=1, \ldots, n$,  such that $F_i$ are defined on the polydiscs $U_i:= \{ |t_k|< \delta_k, |x_j|< \varepsilon _j , k=1, \ldots,\ell,   j=1, \ldots, i \}$;
\item
all roots of $F_{i} (t, x^{i-1}, x_i)=0$, for $(t,x^{i-1}) \in U_{i-1}$,  lie inside the circle of radius $\varepsilon _i$;  
\item
either $F_i(t,0)\equiv 0$ or $F_i \equiv 1$ (and in the latter case 
we define $F_k \equiv 1$ for all $k\le i$).
\end{enumerate}

In practice, to define $F_{i-1}$ in term of $F_i$, we do the following: if $F_i$ is not reduced then we denote by $\Delta_i$ the first nonzero generalized  discriminant of $F_i$  as a polynomial in $x_i$. Then, after a linear change of coordinates in $(t,x^{i-1})$, we may assume that 
\begin{align}\label{eqn:equation}
    \Delta_i(t,x^{i-1})=u_{i-1}(t,x^{i-1})F_{i-1}(t,x^{i-2},x_{i-1}),
\end{align}
$u_{i-1}(0)\ne 0$, where $F_{i-1}$ is a pseudopolynomial in $x_{i-1}$. After this linear change of coordinates, $F_j$, for $j\geq i$, is transformed into a new pseudopolynomial in $x_j$ of degree $p_j$ satisfying again (1), (3) and (4), so it does not affect the form of the previous pseudopolynomials.

An important result due to Varchenko is the following one:

\begin{theorem}[{\rm \cite{varchenko1972, varchenko1973, varchenkoICM, PP}}]\label{thm:Varchenko}
A Zariski equisingular family of singularities is topologically trivial.  
\end{theorem}

\subsection{Construction of a deformation of the germ of an analytic set}

Suppose now that $(V,0)$ is the germ at the origin of an analytic subset of $\K^n$ given by one equation $f(x)=0$.  We do not assume $f$ reduced but we assume $f\not \equiv 0$ . 
We define a local system of coordinates $x=(x_1, \ldots, x_n)$ and a sequence of distinguished pseudopolynomials 
$$f_{i} (  x^i ) = x_i^{p_i}+ \sum_{j=1}^{p_i} a_{i-1,j} (x^{i-1}) x_i^{p_i-j} , \quad i=1, \ldots, n,$$ 
as follows.  

Let, after a local change of coordinates $x$, $f_n$ be the Weierstrass polynomial associated to $f$.  Then we consider the generalized discriminants $\D_{n,i} $ of $f_n$ that are polynomials in the entries of $a_{n-1}:= (a_{n-1,1} , \ldots , a_{n-1,p_{n}} )$.   
  Let $l_n$ be a positive integer such that   
  \begin{align}\label{discriminants:n}
\D_{n,l} ( a_{n-1} )\equiv 0 \qquad l<l_n   ,
\end{align}
and  $\D_{n,l_n}  ( a_{n-1} ) \not \equiv 0$.   Then, after a local change of coordinates $x^{n-1}$, by the Weierstrass Preparation Theorem, we may write 
    \begin{align*}
 \D_{n,l_n} ( a_{n-1} ) =  u_{n-1} (x^{n-1}) \Big (x_{n-1}^{p_{n-1}}
 + \sum_{j=1}^{p_{n-1}} a_{n-2,j} (x^{n-2}) x_{n-1}^{p_{n-1}-j} \Big ) , 
 \end{align*}
 where $u_{n-1}(0)\ne 0$ and for all $j$, $a_{n-2,j}(0)=0$. 
 Note that if $f$ is reduced then $l_n=1$ and the only generalized discriminant we consider is the standard one. 
  Then we set  
$$
f_{n-1} : = x_{n-1}^{p_{n-1}}+
  \sum_{j=1}^{p_{n-1}} a_{n-2,j} ( x^{n-2}) x_{n-1}^{p_{n-1}-j}, $$
 
We continue this construction and  define a sequence of pseudopolynomials $f_{i} (  x^i )$, $i=1, \ldots, n-1$, such that  $f_i= x_i^{p_i}+ \sum_{j=1}^{p_i} a_{i-1,j} (x^{i-1}) x_i^{p_i-j}  $ is the Weierstrass polynomial associated to the first non-identically-zero generalized discriminant $\D_{i+1,l_{i+1}} ( a_{i} )$ of $f_{i+1}$, where we denote in general $a_{i}= (a_{i,1} , \ldots , a_{i,p_{i+1}} )$, 
  \begin{align}\label{eq: polynomials:f_i}
 \D_{i+1,l_{i+1}} ( a_{i} ) =  u_{i} (x^{i})  \Big (x_i^{p_i}+ \sum_{j=1}^{p_i} a_{i-1,j} (x^{i-1}) x_i^{p_i-j}  \Big) ,  
 \quad i=0,\ldots,n-1 ,
\end{align}
and $a_{i-1,j}(0)=0$. 
 Thus, for $i=0,\ldots,n-1,$ the vector  of functions $a_i$ satisfies  
  \begin{align}\label{eq:discriminants_i}
\D_{i+1,l} ( a_{i} )\equiv 0 \text { for } l<l_{i+1}   ,  \quad \D_{i+1,l_{i+1}} ( a_{i} ) \ne 0.
\end{align}
This means in particular that 
  \begin{align*}
\D_{1,l} ( a_{0} ) \equiv 0 \quad \text {for } l<l_1  \text { and }  \D_{1,l_1} ( a_{0} ) \equiv u_0 ,
\end{align*}
where $u_0$ is a nonzero constant. 

Next we apply the algebraic power series version of Theorem \ref{thm:nested_ploski} to the system of equations given by \eqref{eq: polynomials:f_i} and \eqref{eq:discriminants_i}. 
By construction,  this system admits convergent solutions.  Therefore, by Theorem \ref{thm:nested_ploski},   there exist a new set of variables $z=(z_1,\ldots, z_s)$, an increasing function $\tau$,  convergent power series $z_i(x)\in\C \{x\}$  vanishing at $0$, algebraic power series 
$u_i (x^i,z)  \in\C \langle x^{i},z_1,\ldots,z_{\tau(i)}\rangle$,  
and  vectors  
$a_i(x^i,z) \in\C \langle x^{(i)},z_1,\ldots,z_{\tau(i)}\rangle^{p_i}$ 
of algebraic power series, 
such that 
\begin{itemize}
 \item
$ z_1(x),\ldots, \, z_{\tau(i)}(x) \text{ depend only on }(x_1,\ldots, x_{i})$,
\item
 $ a_i(x^i,z),\, u_i (x^i,z)  \text{ are solutions of  \eqref{eq: polynomials:f_i} and \eqref{eq:discriminants_i} }$, 
\item 
$ a_i(x^i)= a_i(x^i,z(x^i)),  
u_i(x^i)= u_i(x^i,z(x^i))$.  
\end{itemize}
It is essential that the new solutions $u_i(t,x^i), a_{i,j}(t,x^i)$ depend only on the first $i$ variables in $x$.  Thanks to this property it is easy to check that the one parameter deformation $t\to \{(t,x); F(t,x)=0\}$ of 
$$
 F(t, x) = x_n^{p_n}+ \sum_{j=1}^{p_n} a_{n-1,j} ( x^{n-1}, tz(x^{n-1})) x_n^{p_n-j}  
$$
is Zariski equisingular.  
Because $F(1,x)=f(x)$ and $F(0,x)$ is algebraic, we have shown 

\begin{theorem}[\cite{mostowski, BKPR, BPR}]\label{thm:ZariskitoNash}
Every analytic set germ given by one equation $f=0$, $f\in \K\{x\}$, is Zariski equisingular to a germ defined by an algebraic power series.      
\end{theorem} 

\begin{remark}
The first proof of algebraicity of analytic set germs was given by Mostowski  \cite{mostowski}.  
At that time Popescu's theorem was not yet available.  Instead, Mostowski proposes an ingenious  recursive construction  of the system of equations \eqref{eq: polynomials:f_i} and \eqref{eq:discriminants_i} giving Zariski equisingularity conditions by local linear changes of coordinates and,   at the same time, step by step,  provides the deformation to an algebraic power series solution following the recipe given by P{\l}oski \cite{ploski1974}. 
\end{remark} 

\subsection{Algebraization.  Final steps.}

Using Varchenko's theorem (Theorem \ref{thm:Varchenko}), Theorem \ref{thm:ZariskitoNash} shows that every analytic set germ is homeomorphic to the germ of an algebraic power series, and moreover by an ambient homeomorphism.  This holds not only for hypersurfaces. 
 If $(V,0)$ is given by a finite system of equations $g_s=0, s=1, \ldots,k$, $g_s\in \K\{x\}$, then we proceed as follows.  In a system of local coordinates we replace $g_s=0$ by the associated pseudopolynomials:  
   \begin{align*}
g_{s} ( x)= x_n^{r_s}+ \sum_{j=1}^{r_s} a_{n-1,s,j} (x^{n-1}) x_n^{r_s-j} ,
\end{align*}
and arrange the coefficients $a_{n-1,s,j}$ in a row vector  
$a_{n-1} \in \K\{x^{n-1}\}^{p_n}$ where $p_n:=\sum_s r_s$.  
Then apply the previous construction to $f_n$ being the product of the $g_s$'s. 
After solving the system of equations for $t\in \K$ we define 
\begin{align*}
 F_n(t,x) = \prod_s G_s(t,x)  , \quad G_{s} (t, x)=  x_n^{r_s}+ 
 \sum_{j=1}^{r_s} a_{n-1,s,j} (x^{n-1}, t z(x^{n-1})) x_n^{r_s-j}  . 
\end{align*}
   Then Varchenko's proof of Theorem \ref{thm:Varchenko} provides a topological trivialization that preserves the zero sets of $G_s(x,t)=0, s=1, \ldots,k$, and thus the ambient homeomorphism between $(V,0)$ and $\{G(x,0)_s=0, s=1, \ldots,k\}$. 

Finally,  the following theorem due to Bochnak and Kucharz \cite{bochnakkucharz1984}, based on  Artin-Mazur Theorem,  gives an equivalence, up to a Nash diffeomorphism, between the zeros of algebraic power series (or equivalently germs of Nash functions) and the local zeros of polynomial functions. 

 \begin{theorem}\label{thm:bochnakkucharzbetter}
Let $\K = \R$ or $\C$.  Let  $g_s : (\K^{n},0)\to (\K, 0)$,  be a finite family of Nash function germs.  
Then there is a Nash diffeomorphism $ h : (\K^n,0) \to (\K^n,0)$ and analytic (even Nash) units 
$u_s : (\K^{n},0)\to \K$, $u_s(0)\ne 0$, such that for all $s$,  $u_s(x) g_s (h(x))$  are 
germs of polynomials.    
\end{theorem}  

\subsection{Algebraization of function germs.} \label{ssec:functions}

The algebraization of analytic set germs can be extended to analytic function germs, 
that is, mappings with values in $\K$.  The following theorem was proven in \cite{BPR}.  
 
\begin{theorem} \label{homeotopolynomial}  
Let  $\K = \R$ or $\C$.  
 Let  $g: (\K^n,0)\to (\K, 0)$ be an analytic function germ.  
 Then there is a homeomorphism $\sigma : (\K^n,0) \to (\K^n,0)$ such that $g\circ \sigma$ 
 is the germ of a polynomial.  
\end{theorem}

The idea of adapting Zariski equisingularity to the function case comes from \cite{varchenko1972}.  
Given a family of analytic function germs 
$g_t(y) = g(t, y_1, \ldots , y_{n-1}):(\K ^{n-1},0)\to (\K, 0)$  parameterized by $t\in \K^\ell$, the idea is to consider the associated family of set germs defined by the graph of $g$, the zero set of $F(t,x_1, \ldots, x_n) := x_1 - g(t, x_2, \ldots , x_{n})$.  If the family 
$V=V(F)$ is Zariski equisingular, with respect to the system of coordinates 
$x_1, \ldots, x_n$, then the trivialization constructed in \cite{varchenko1972} does not move the variable $x_1$
\begin{align}\label{eq:preservex_1}
h_t (x_1, \ldots x_n) = (x_1 , \hat h_t (x_1, x_2, \ldots, x_n)).   
\end{align} 
Set $\sigma_t (y) := \hat h_t (g(y), y) $.  
Then 
$$
g_t\circ \sigma_t = g_{t_0}, 
$$  

To complete the passage from algebraic power series to polynomials one uses again Theorem 
\ref{thm:bochnakkucharzbetter}.  See \cite{BPR} for details. 



\section{Algebraization of the germ of a meromorphic function in dimension 2}
As said before, every two variables analytic function germ is weakly finitely determined, so it is analytically equivalent to a polynomial function germ. The case of meromorphic function germs has been studied in \cite{CM}, where it is shown that, in contrast to the  analytic function germs, a two variable meromorphic germ is not always weakly finitely determined. But the question whether a two variable meromorphic function germ is analytically equivalent to a rational function germ remains open. The following result has been proved recently: 

\begin{theorem}\cite{GM}
Let $\K=\R$ or $\C$. Let $\phi$ be the germ of a meromorphic function at $(\K^2,0)$. Then there is the germ of an analytic diffeomorphism $h:(\K^2,0)\lgw (\K^2,0)$ such that $\phi\circ h$ is the germ of an algebraic meromorphic function, that is, $\phi\circ h={f}/{g}$ where $f$, $g\in\C\langle x_1,x_2\rangle$.\\
Moreover, for any $k\in\N$, we may assume that $h(x)-x\in(x)^k$.
\end{theorem}
\begin{proof}[Sketch of proof]
The main case is $\K=\C$. We consider this case now. The idea is to associate to $\phi$ a germ of analytic 1-form $\omega$, and then to use 
P{\l}oski theorem to construct an analytic deformation of $\o$ that has one algebraic fiber. Then we use a result of D. Cerveau and J.-F. Mattei to prove that this deformation is 
analytically trivial. Let us give more details.\\

{\bf Step 1.} Let us write $\phi={f}/{g}$ where $f$ and $g$ are coprime convergent power series. Let us write
$$f=f_1^{\ell_1}\cdots f_p^{\ell_p}\text{ and } g=g_1^{k_1}\cdots g_q^{k_q}$$
where the $f_i$ and $g_j$ are irreducible convergent power series and the $\ell_i$ and $k_j$ are positive integers.
We set
$$\theta:=\frac{f_1\cdots f_p g_1\cdots g_q}{fg}(gdf-fdg).$$
Then $\theta$ is a holomorphic 1-form and each of its analytic divisors $h\in\C\{x\}$ is coprime to $fg$ (here $x=(x_1,x_2)$). Let us recall that a \emph{divisor of $\theta$} is a common divisor of the coefficients of $dx_1$ and $dx_2$ in the expansion of $\theta$. Then we can prove the following lemma:
\begin{lemma}\cite[Prop. 3.2]{GM}\label{lem:GM1}
Let $h\in\C\{x\}$ be an irreducible convergent power series. Then $h$ divides $\theta$ if and only if there is $c\in\C$ such that $h$ divide $f-cg$. In this case, if $\mu$ is the largest power of $h$ dividing $\theta$, then $\mu+1$ is the largest power of $h$ dividing $f-cg$.
\end{lemma}

Denote by $h_1$, \ldots, $h_e$ the irreducible divisors of $\theta$, and by $\mu_1$, \ldots, $\mu_e\in\N^*$ the maximal exponents such that $h_1^{\mu_1}\cdots h_e^{\mu_e}$ divides $\theta$. We define 
$$\o:=\frac{1}{h_1^{\mu_1}\cdots h_e^{\mu_e}}\theta.$$ This is a holomorphic 1-form with (at most) an isolated singularity (let us recall that the singular locus of $\o$ is the zero locus of the coefficients of $dx_1$ and $dx_2$ in $\o$). We recall that a \emph{first integral of $\o$}  is a meromorphic function $R$ such that $\o\wedge dR=0$ ;  so $\phi$ is a first integral of $\o$.

We denote by $c_1$, \ldots, $c_e\in\C$ the complex numbers corresponding to the $h_i$ in Lemma \ref{lem:GM1}, that is, $f-c_ig=h_i^{\mu_i+1}\rho_i$ for some convergent power series $\rho_i$.\\

{\bf Step 2.} 
We consider the following system of equations:
\begin{equation}\label{equations}\tag{S}
\left\{\begin{array}{ccc}
y_{1,1}^{\ell_1}\cdots y_{1,p}^{\ell_p}-c_1y_{2,1}^{k_1}\cdots y_{2,q}^{k_q}&=&y_{3,1}^{\mu_1+1}y_{4,1}\\
y_{1,1}^{\ell_1}\cdots y_{1,p}^{\ell_p}-c_2y_{2,1}^{k_1}\cdots y_{2,q}^{k_q}&=&y_{3,2}^{\mu_2+1}y_{4,2}\\
\vdots & & \\
y_{1,1}^{\ell_1}\cdots y_{1,p}^{\ell_p}-c_ey_{2,1}^{k_1}\cdots y_{2,q}^{k_q}&=&y_{3,e}^{\mu_e+1}y_{4,e}
\end{array}\right.
\end{equation} 
By assumption
$$y(x):=(f_1(x), \ldots, f_p(x), g_1(x),\ldots, g_q(x), h_1(x), \ldots, h_e(x), \rho_1(x),\ldots, \rho_e(x))$$
is a solution of \eqref{equations}. By Theorem \ref{thm:ploski-app}, there exists a vector of algebraic power series $y(x,z)\in\C\langle x,z\rangle$, a solution of \eqref{equations}, and convergent power series $z_1(x)$, \ldots, $z_s(x)\in\C\{x\}$ such that
$$y(x)=y(x,z(x)).$$
We denote by
$$f_1(x,z), \ldots, f_p(x,z), f_1(x,z),\ldots, g_q(x,z), h_1(x,z), \ldots, h_e(x,z), \rho_1(x,z),\ldots, \rho_e(x,z)$$
the components of $y(x,z)$. Let $k_0\in\N$. For $t\in\left[0,1\right]$, we set
$$z(x,t)=z_{k_0}(x)+(1-t)r_{k_0}(x)$$
where $z_{k_0}(x)$ is the truncation of $z(x)$ at order $k_0$ and $r_{k_0}(x)=z(x)-z_{k_0}(x)$. So $z(x,0)=z(x)$ and $z(x,1)$ is a polynomial. We set
$$F_i(x,t):=f_i(x,z(x,t)),\ G_j(x,t):=g_j(x,z(x,t)),\ H_k(x,t):=h_k(x,z(x,t))$$
and
\begin{align*}
 & F(x,t):=F_1(x,t)^{\ell_1}\cdots F_p(x,t)^{\ell_p},\ G(x,t):=G_1(x,t)^{k_1}\cdots G_q(x,t)^{k_q}, \\ 
 & H(x,t)=H_1(x,t)^{\mu_1}\cdots H_e(x,t)^{\mu_e} , \\ 
& \Phi(x,t):=\frac{F(x,t)}{G(x,t)}.  
\end{align*}
We see that $\Phi(x,0)=\phi(x)$, and $\Phi(x,1)$ is an algebraic meromorphic function germ. We set 
$$\Theta:=\frac{F_1\cdots F_p G_1\cdots G_q}{FG}(GdF-FdG)\ \text{ and }\  \Omega:=\frac{\Theta}{H_1^{\mu_1}\cdots H_e^{\mu_e}}.$$

{\bf Step 3.} We denote by $\Omega_\tau$ the restriction of $\Omega$ to the plane of equation $t=\tau$.  Then $\Omega_0=\o$ has an isolated singularity at 0 and $\Omega_1$ is a 1-form having $\Phi(x,1)$ as a first integral.\\
In fact, one can prove that $\Omega_t$ has an isolated singularity for any $t\in[0,1]$ if $k_0$ is chosen large enough (see \cite{GM} for more details). This comes from the fact that the coefficients of $dx_1$ and $dx_2$ in $\Omega_t$ coincide with those of $\Omega_0=\o$ up to order $\geq k_0$. Then we use the following lemma:

\begin{lemma}\cite[Lemme 2.1, p. 149]{CM}
Let $\o$ be an analytic 1-form germ with an isolated singularity at $0\in\C^n$. Then there is an integer $N$ such that, for every integrable analytic 1-form germ $\Omega$ at $(\C^{n+m},0)$ satisfying
\begin{enumerate}
\item $\Omega_0=\o$,
\item $\Omega-\o\in (x)^N$,
\end{enumerate}
there is a germ of biholomorphism $\Psi$, of the form $\Psi(x,t)=(\Psi_1(x,t),t)$, and a unit $u\in\C\{x,t\}$, such that
$$\Psi_*\Omega=u\o.$$
\end{lemma}

Therefore, if $k_0\geq N$, we have that $\Psi_*\Omega=u\o$ and $\phi\circ\Psi_1(x,1)$ is a first integral of $\Omega_1$.  
Now we use \cite[Thm 1.1, p. 137]{CM} which asserts that the set of first integrals of a meromorphic 1-form having a meromorphic first integral has the form 
$$\C(\kappa)=\{\gamma\circ\kappa\mid \gamma\in \C(T)\}.$$
We can apply this to $\Omega_1$: let $\kappa$ be such a meromorphic function, so $\Phi(x,1)=R(\kappa)$ for some rational function $R\in\C(T)$. Since $\Phi(x,1)$ is algebraic, it satisfies $P(x,\Phi(x,1))=0$ for some nonzero polynomial $P(x,y)$. So $\kappa$ annihilates $P(x,R(z))=0$. But $P(x,R(z))$ is a rational function, so $\kappa$ is a root of its numerator, thus $\kappa\in\C\langle x\rangle$. In particular, all the first integrals of $\Omega_1$ are algebraic. But $\phi\circ\Psi_1(x,1)$ is also a first integral of $\Omega_1$, so $\phi\circ\Psi_1(x,1)$ is algebraic. This proves the theorem for $\K=\C$ with $h(x):=\Psi(x,1)$.\\

The real case follows essentially in the same way, but we have to prove that at each step of the proof, the objects that intervene are real.
\end{proof}

\end{document}